\documentclass[10pt]{article}
\usepackage[utf8]{inputenc}
\usepackage[usenames,dvipsnames,svgnames,table]{xcolor}	
\usepackage[]{amsmath}
\usepackage{cases}
\usepackage{empheq}
\usepackage[colorlinks=true]{hyperref}
\hypersetup{colorlinks,
citecolor=brown, linkcolor=blue, linktocpage
}
\usepackage{cleveref} 
\usepackage{amssymb}
\usepackage{bbm}
\usepackage{enumerate}   
\usepackage{geometry}
\usepackage{amsfonts}
\usepackage{color}
\usepackage[]{amsthm} 
 \usepackage[shortlabels]{enumitem}

\newtheorem{theorem}{Theorem}[section]
\newtheorem{lemma}[theorem]{Lemma}
\newtheorem{cor}[theorem]{Corollary}
\newtheorem{proposition}[theorem]{Proposition}

\newtheorem{remark}[theorem]{Remark}

\def\E{\mathbb{E}}
\def\R{\mathbb{R}}

\def\P{\mathbb{P}}

\def\N{\mathbb{N}}

\def\1{\mathbbm{1}}

\def\D1{\frac{\partial}{\partial x}}

\newcommand{\Ee}{\mathcal E}
\newcommand{\Ff}{\mathcal F}

\newcommand{\Hh}{\mathcal H}

\newcommand{\Ii}{\mathcal I}

\newcommand{\Ss}{\mathcal S}

\newcommand{\Vv}{\mathcal V}

\newcommand{\Zz}{\mathcal Z}
\newcommand{\guillaume}[1]{\textcolor{cyan}{#1}}

\usepackage{xspace}

\def\restriction#1{\raise-.5ex\hbox{\ensuremath|}_{#1}}

\title{Zeros of the Brownian Sheet}
\author{Keming Chen \footnote{PROB, MATH, SB, EPFL, CH-1015 Lausanne, Switzerland; keming.chen@epfl.ch} , Guillaume Woessner \footnote{ PROB, MATH, SB, EPFL, CH-1015 Lausanne, Switzerland; guillaume.woessner@epfl.ch.}}
\date{}

\begin{document}
\maketitle
\noindent
\textbf{Abstract:} In this work we firstly answer to a question raised by Khoshnevisan in \cite[Open Problem 4]{khoshnevisan2007slices} by proving that almost surely there is no projection of big enough rank changing the Hausdorff dimension of the zeros of the Brownian sheet. Secondly, we prove that almost surely for every projection whose rank isn't matching the aforementioned condition, the projection of the zero set is the entirety of the projective space.
\textbf{Classification:} 60G15, 60G17, 60G60


\section{Introduction}
\subsection{Notations and definitions} 
\label{section_1}
We study some fractal properties of the zero set of the Brownian sheet.\\
For any positive integer $N$, an $N$-parameter Brownian sheet is a continuous Gaussian process $W~:~\R_+^N\to\R$ such that the finite dimensional marginal distributions of $W$ are multivariate centred normal distributions with covariance satisfying, for all $s=(s_1,\dots,s_N)$ and $t=(t_1,\dots,t_N)$ in $\R_+^N$,
\begin{align}
\label{def_BS_1d}
\E[W(t)W(s)]=\prod_{i=1}^N \min(s_k,t_k).
\end{align}
For any positive integer $d$, a $d$-dimensional, $N$-parameter Brownian sheet $W=(W_1,\dots,W_d)$ is a vector consisting of $d$ independent $N$-parameters Brownian sheets. For the sake of simplicity we will often simply write $(N,d)$-Brownian sheet, or even Brownian sheet when it does not lead to any confusion. When $N=1$, this process is the well-known standard Brownian motion in $\R^d$. Throughout this work, $W$ will denote a $d$-dimensional, $N$-parameter Brownian sheet. \\
In this paper $\Zz:=W^{-1}(0)=\{t\in\R_+^N~:~W(t)=0\}$, the set of the zeros of the Brownian sheet. We will be studying fractal properties of $\Zz$. For any Borel set $E$ of a space $\R^n$, $\dim_H(E)$ will denote the Hausdorff dimension of the set $E$, and $\mathcal{H}_s(E)$ the associated Hausdorff measures. For a further discussion of Hausdorff measures and Hausdorff dimensions, please refer to Taylor's text book \cite{taylor1986measure}. We would like to emphasize one elementary but important property of the Hausdorff dimension though, which is that for every given linear projection map $p$ on $\R^n$ and any Borel set $E\subset \R^n$ we have that $\dim_H(p(E))\leq \dim_H(E)$.\\
~\\
Moreover, one particular point worth noting before presenting our results is the following. If we define for $N=2$
$$
U(s, \cdot)=e^{-s/ 2} W\left(e^{s}, \cdot\right),
$$
then $U(s,\cdot)$ is an Ornstein-Uhlenbeck process on the Wiener space of continuous functions $\mathcal{C}(\R_+,\R^d)$, see \cite[Chapter 1]{Walsh} for more details. For $d>1$ and any $t>0$, with Fubini one can show that the set $\{s>0: U(s,t)=\overrightarrow{0}\}$ is empty almost surely, but it doesn't follow that almost surely, $\forall t>0, \ \{s>0: U(s,t)=\overrightarrow{0}\}$ is empty. This difference between an uncountable family of statements each of them holding almost surely, and the statement that almost surely every statement holds simultaneously is a big problem that we will have to deal with in all this paper. Indeed, one of the objective of this paper is to compare the Hausdorff dimensions of $\Zz$ and $p(\Zz)$ simultaneously for an uncountable family of projections $p$. While we will first do the comparison for a fixed single projection, to obtain a result for every projection $p$ simultaneously will require some additional work.\\
Finally, we will make use of the theorem below, proven in \cite[(5.14)]{ehm1981sample},
\begin{proposition}(Ehm 1981)\\
\label{Ehm}
	For a $(N,d)$-Brownian sheet $W$, 
$$\dim_H\{\Zz\}=\left(N-\dfrac{d}{2}\right)^+, \qquad as, $$
where $x^+$ denotes $\max(x,0)$ for any $x\in\R$.
\end{proposition}

\subsection{Literature and main results}
In \cite{ehm1981sample} the author proved that for the Brownian motion (i.e. when $N=1$) if $d\geq 2$, almost every path does not have any zeros. In \cite{kakutani1944131} it is shown that almost every paths were self-avoiding if and only if $d\geq 4$. Those works has been improved using the notion of "quasi-everywhere" which is stronger than "almost-everywhere" : a property $P$ concerning the paths $\omega$ is said to hold quasi-everywhere if the capacity of the set $\{\omega~:~\omega \text{ has not the property }P\}$ is null (see \cite{fukushima1984basic} for more details on the capacity). If $P$ holds quasi-everywhere, then it holds almost-everywhere. In \cite{lyons1986critical} the author showed that quasi every paths is self-avoiding if and only if $d\geq 6$, and in \cite{fukushima1984basic} (see also \cite{orey1973sample}) the author showed that quasi every path has no zero if and only if $d\geq 4$. In comparison,\cite{kono19844} showed that the capacity of the set of paths approaching arbitrarily close to 0 is positive if $d\leq 4$.\\
As said above, the set of zeros $\Zz$ has been widely studied in the literature. Also, the links between the Hausdorff dimensions of a set and its projection into a given direction has been determined in \cite{marstrand1954dimension} where the author proved that in dimension $2$ the Hausdorff dimension of a Borel set $E$ is unchanged by almost any orthogonal projection. This result led Khoshnevisan in \cite{khoshnevisan2007slices} to prove that, if $N=2$ for any orthogonal projection $p$ onto an axis, the set $\Zz$ and the set $p(\Zz)$ almost surely share the same Hausdorff dimension. Our main objective is to study a generalization of this theorem, which was stated as an Open Problem 4 in the same Khoshnevisan's work \cite{khoshnevisan2007slices}. It questioned if there existed any (necessarily random) direction onto which the projection of the set $\Zz$ was not equal to $\dim_{H}(\Zz)$. The following theorem answers Open Problem 4.
\begin{theorem}[Open problem]
\label{thm0}
For $N=2$ and $d=2 \text{ or } 3$, almost surely, for any non-zero linear map $p~:~ \R^2\rightarrow \R$, 
\begin{align}
\label{thm0_eq}
dim_H[p(\Zz)]=\dim_H(\mathcal{Z}).
\end{align}
\end{theorem}

\noindent As we said in the Introduction, it is noticeable that the part "almost surely" in Theorems \ref{thm0} (and below in Theorem \ref{thm1}) is before the quantification of the linear maps. We will have to prove in a first part that the result in \eqref{thm0_eq} hold for any fixed linear maps $p$ almost surely, which is not itself an obvious result, and then in a second part deduce that the results hold almost surely for all linear maps $p$ simultaneously. One should note that the contribution made in our theorem is qualitatively similar to the framework exposed in the end of Section \ref{section_1} : we go from a statement holding for a fixed projection almost surely to a statement holding almost surely for every projection.\\
\noindent Actually we will even go a bit further, and after proving Theorem \ref{thm0}, we will also prove the following generalization to any dimension of Theorem \ref{thm0} separately

\begin{theorem}
\label{thm1}
For every $N>\frac{d}{2}$, almost surely, for any orthogonal projection $p: \R^N\rightarrow \R^{N'}$,  with $N-\frac{d}{2}\leq N'<N$ where $N':=rank(p)$, 
\begin{align}
\label{thm1_eq}
dim_H[p(\Zz)]=\dim_H(\mathcal{Z}).
\end{align}
\end{theorem}

\begin{remark}~
\label{remark}
\begin{itemize}
\item[$\bullet$] Of course, the result in Theorem \ref{thm0} is contained in Theorem \ref{thm1}. The reason why we chose to prove the former is to introduce some notions and arguments in a simpler framework than the latter. The main difficulty we will encounter in Theorem \ref{thm1} specifically is that Remark \ref{remark_renormalization} below does not hold anymore in higher dimensions (see \cite{cairoli-walsh}), such that we have to introduce a new parametrization in Remark \ref{remark_parametrization} instead.
\item We do not consider the case $N\leq \frac{d}{2}$, because otherwise we have immediately with Proposition \ref{Ehm} that $\dim_H(p(\Zz))\leq\dim_H(\Zz)=0$.
\item Similarly, we do not consider the case $N>\frac{d}{2}$ and $N'<N-\frac{d}{2}$, because otherwise we have with Proposition \ref{Ehm} that
$$\dim_\mathcal{H}(p(\mathcal{Z})) \leq N' < N-\dfrac{d}{2} = \dim_\mathcal{H}(\mathcal{Z}).$$
\item With Ehm's Proposition \ref{Ehm}, we can explicitly compute the value of $\dim_H(p(\Zz))$ in \eqref{thm0_eq} and \eqref{thm1_eq}.
\item To prove Theorem \ref{thm1} we will need to prove an intermediate result given in \eqref{prop_kaufman_1_cor}, which states that the Brownian sheet when $N\leq \frac{d}{2}$ doubles the Hausdorff dimensions, and generalizes results from Kaufman in \cite{kaufman1969propriete} for $N=1$ and Mountford in \cite{mountford1989uniform} for N=2. 
\end{itemize} 
\end{remark}

\noindent Now we want to study the case $N>\frac{d}{2}$ and $N'<N-\frac{d}{2}$ . By Remark \ref{remark}, the result from Theorem \ref{thm1} cannot still hold. Although, it turns out that we can prove the following theorem, which means that \textit{simultaneously} for every projection, every point of the projective space is a projected zero.
\begin{theorem}
\label{thm2}
    For every $N>\frac{d}{2}$, almost surely, for any orthogonal projection $p: \R^N\rightarrow \R^{N'}$,  with $N'<N-\frac{d}{2}$ where $N':=rank(p)$,
    \begin{align}
        \label{thm2_eq}
        p(\Zz)=\R_+^{N'}.
    \end{align}
\end{theorem}

\noindent The rest of this paper is organized as follows. In Section \ref{section_thm0} we will give a proof for the specific case of Theorem \ref{thm0}. In Section \ref{section_thm1} we will prove the general Theorem \ref{thm1}. In Section \ref{section_thm2} we will prove the Theorem \ref{thm2}. In the Annex, one can find some small but useful lemmas used in the proofs of Theorem \ref{thm0} and Theorem \ref{thm1}.

\newpage

\section{Proof of Theorem \ref{thm0}}
\label{section_thm0}
We will present the proof in the case $d=2$, but one can easily see that everything works the same in the case $d=3$. Also, to compute the Hausdorff dimension of a given set $A$ of a space $\R^n$, it is enough to compute the Hausdorff dimension of $A\cap C$ for every bounded hypercube of $\R^n$. Thus, for notational simplicity we will only work on the square $[1,2]^2$ of $\R^2$, \textit{ie} we will consider that the Brownian sheet is only defined on $[1,2]^2$, unless the opposite is explicitly stated. The case of the other squares of the form $[a,b]^2$ for $0<a<b$ can be done the in same way. Additionally, we can suppose without loss of generality throughout this section that $\Zz\cap[1,2]^2\neq\emptyset$, since we can always take compact subset of $\R_+^N$ big enough to intersect $\Zz$. Finally, one should note that the problems of definitions that may occur at the boundary of $[1,2]^2$ can be ignored thanks to the fact that our proof works for any bounded square $C$.\\
~\\
First we will fix $p~:~\R^2\to\R$ a linear map such that $\ker(p)\cap\dot{\R}^2_+\neq\{0\}$ (\textit{ie} such that the zeros of $p$ are the graph of an increasing linear function). Secondly we will explain how to adapt the result to the other linear maps. Finally, we will show that the result still holds for every linear map \textit{simultaneously}.

\paragraph{Notations}~\\
Let $E:=p([1,2]^2)$ and for a given integer $n>0$, we have that $E$ is a segment that we can cover it by a family of disjoint intervals of length $2^{-n}$ and denote this family of intervals as $\Ii^n:=(I_i^n)_{i}$, (one should note that $\#\Ii^n\leq c 2^{n+1}$ for a certain constant $c(p)$). Then $p^{-1}(I_i^n)$ is a "tube" in $[1,2]^2$, and we will denote the bounded line in the middle of this tube as $\Gamma^n_i$. Let $\alpha,\beta^n_i$ be the real numbers such that $\Gamma^n_i \subset \{(t_1, \alpha t_1+\beta^n_i) ~:~t_1\geq 0\}$ (one should note that $\alpha$ does not depend on $n$ nor $i$). We can partition the tube $p^{-1}(I_i^n)$ into a family of disjoint tilted squares $\Ss^n_i := (S_{ij}^n)$, and we say that a square $S_{ij}^n$ is \textit{good} if $\Zz\cap S_{ij}^n \neq \emptyset$. We also say that an interval $I^n_i$ is \textit{good} if there exists a good square $S^n_{ij}$ in $p^{-1}(I^n_i)$.\\
Moreover, $|\cdot|$ will always denote the usual euclidean norm for the space it will be defined on, and $B(x,r)$ will always denote the associated open ball centered in $x$ and of radius $r$.\\

\subsection{The Kaufman strategy}

Let's fix $p~:~\R^2\to\R$ a linear map such that $\ker(p)\cap\dot{\R}^2_+\neq\{0\}$ (\textit{ie} such that the zeros of $p$ are the graph of an increasing linear function).\\
The keystone of our proof is the following proposition, which will imply almost directly the Theorem \ref{thm0}.

\begin{proposition}
\label{Prop_Kaufman_0}
There exists a (random) rank $n_0$ such that, for any $n\geq n_0$ and any $i$, there are at most $n^7$ good squares $S_{ij}^n$ in $p^{-1}(I^n_i)$.
\end{proposition}

\begin{remark}
\label{remark_renormalization}
Thanks to Lemma \ref{lem_BM_line} we know that, up to a regular renormalization of time, the Brownian sheet along any line $\Gamma^n_i\subset [1,2]^2$ is a Brownian motion. Moreover, with Lemma \ref{lem_BM_line} we can see that the renormalization that we will note here by $\varphi^n_i$ is such that $m<|\varphi'^n_i|<M$ for certain $m,M>0$ independent of $n$ and $i$. This allows us to simplify the proof of Proposition \ref{Prop_Kaufman_0} while assuming that the Brownian sheet along any line $\Gamma^n_i\subset [1,2]^2$ is indeed a Brownian motion. Thus, the reader will not be surprised if we denote the Brownian sheet along any line $\Gamma^n_i\subset [1,2]^2$ by the abuse of notation $B^n_i(t)$, where $t$ stands for a parametrization of the line $\Gamma^n_i$, and assume that it is a Brownian motion (not necessarily of initial value 0).
\end{remark}

\noindent To prove Proposition \ref{Prop_Kaufman_0}, it is natural to study the moments when the process along the line $\Gamma^n_i$ is small. Thus, it is natural to define the stopping times as follows.\\
For $n$ and $i$ fixed, and $R>1$, we define for $a\in\R^2$ by iteration on $k\in\N^\times$,
\begin{align*}
T_R &:= \inf\{t\geq 0 ~:~|B_i^n(t)|\geq R\},\\
\tau_1(a\, ;n,i) &:= \inf\{t\geq 0 ~:~|B_i^n(t)-a|\leq 2^{-n/2}\},\\
\tau_{k+1}(a\, ;n,i) &:= \inf\{t\geq \tau_k+2^{-n} ~:~|B_i^n(t)-a|\leq 2^{-n/2}\}.
\end{align*}
\noindent One should note that, when the context does not lead to any confusion, we will drop the variables when we write $\tau_k(a,n,i)$.\\
Now, before proving Proposition \ref{Prop_Kaufman_0}, we will have to state and prove two intermediate lemmas. Let us announce them.
\begin{lemma}
\label{lemme_thm0_1}
There exists a constant $c>0$ such that, for any $n\geq1$ and any line $\Gamma^n_i$,  we have that, for all $a\in\R^2$ such that $B(a,2^{-n/2})\subset B(0,R)$, and for all $k\geq 1$, 
$$\P(\tau_k(a\,;n,i) < T_R)<e^{-\frac{ck}{n}}.$$
\end{lemma}

\begin{proof}
We choose $a\in\R^2$ such that $B(a,2^{-n/2})\subset B(0,R)$, and we compute, for $k\geq 1$,
\begin{align*}
\P(\tau_k<T_R) &= \P(\tau_k<T_R ~|~\tau_{k-1}<T_R)\P(\tau_{k-1}<T_R) + 0\\
&=\dots\\
&=\P(\tau_1<T_R)\prod_{l=2}^k \P(\tau_l<T_R~|~\tau_{l-1}<T_R).
\end{align*}
Then, we have, for $l=2,\dots,k$,
\begin{align}
\label{proof_lemme_thm0_1}
\P(\tau_l > T_R~|~\tau_{l-1}<T_R) \geq &~\P(\tau_l > T_R ~|~\tau_{l-1}<T_R,~|B^n_i(\tau_{l-1}+2^{-n}~)-a|>2^{-\frac{n}{2}+1})\\
\nonumber
\times &~\P(|B^n_i(\tau_{l-1}+2^{-n}~)-a|>2^{-\frac{n}{2}+1}~|~\tau_{l-1}<T_R).
\end{align}
For the second factor of the product \eqref{proof_lemme_thm0_1}, by the strong Markov property, we know that,
\begin{align*}
\P(|B^n_i(\tau_{l-1}+2^{-n}~)-a|>2^{-\frac{n}{2}+1}~|~\tau_{l-1}<T_R) \geq \inf_{x\in C_0} \P_x(|B^n_i(2^{-n}~)|>2^{-\frac{n}{2}+1})=:c_1>0,
\end{align*}
where $C_0$ is the square centered in $0$ of side length $2^{-n/2}$. From the fact that $\P(B(t)>2\sqrt{t})$ is independent of $t$, we have that $c_1$ is a constant not depending on $n$ nor $i$ nor $a$, but only on $R$.\\
For the first factor of the product \eqref{proof_lemme_thm0_1}, we can use Lemma \ref{lem_Davis} to compute
\begin{align*}
&\P(\tau_l > T_R ~|~\tau_{l-1}<T_R,~|B^n_i(\tau_{l-1}+2^{-n}~)-a|>2^{-\frac{n}{2}+1})\\
\geq~ &\P(\tau_l > T_R ~|~\tau_{l-1}<T_R,~|B^n_i(\tau_{l-1}+2^{-n}~)-a|=2^{-\frac{n}{2}+1})\\
=~&\dfrac{\ln 2}{\ln (R2^{n/2})}\geq \dfrac{c_2}{n},
\end{align*}
for a certain $c_2>0$ not depending on $n$ nor on $i$ nor $a$.\\
The results above implies that there exists a certain $c>0$, not depending on $n$ nor $i$ and nor $a$, such that
$$\P(\tau_k<T_R) \leq \left( 1-\dfrac{c_1c_2}{n} \right)^{k-1} \leq e^{-\frac{ck}{n}}.$$
\end{proof}
\noindent In Lemma \ref{proof_lemme_thm0_1}, we controlled the moments when the process along the line $\Gamma^n_i$ in a small square of length $2^{-n/2}$. But, in view of Lemma \ref{lem_OreyPruitt}, one could be surprised, since we are more interested in the moments when the process along the line $\Gamma^n_i$ is less or equal to $n2^{-n/2}$. This problem is solved by the next Lemma. In this view, it is natural to define the following stopping times.\\
We define by iteration on $k\in\N^\times$,
\begin{align*}
\tau_1'(n,i) &:= \inf\{t\geq 0 ~:~|B_i^n(t)|\leq n2^{-n/2}\},\\
\tau_{k+1}'(n,i) &:= \inf\{t\geq \tau_k'(n,i)+2^{-n} ~:~|B_i^n(t)|\leq n2^{-n/2}\}.
\end{align*}
One should note that, when the context does not lead to any confusion, we will allow us not to write the variables of the $\tau_k'$.

\begin{lemma}
\label{lemme_thm0_2}
Let $R>1$. Almost surely, there exists a (random) rank $n_0$ such that for all $n \geq n_0$,
$$\forall I^n_i\in\Ii^n~:~\tau_{n^7}'(n,i)\geq T_R.$$
\end{lemma}

\begin{proof}
We define, for $n\in\N^\times$,
$$p_n:=\P(\exists I^n_i\in\Ii^n~:~\tau_{n^7}'(n,i)<T_R).$$
We know that for $n$ big enough, the ball $B(0,n2^{-n/2})$ can be covered by $n^3$ smaller balls $B(a_r,2^{-n/2})\subset B(0,R)$. We also know that if the ball $B(0,n2^{-n/2})$ is visited more than $n^7$ times by the process $B^n_i$ before leaving $B(0,R)$, then it means that one of the balls $B(a_r,2^{-n/2})$ is visited at least $n^4$ times by the process $B^n_i$ before leaving $B(0,R)$. This allows us to write, using Lemma \ref{lemme_thm0_1} at the end
\begin{align*}
p_n &= \P(\exists I^n_i\in\Ii^n~:~\tau_{n^7}'(n,i)<T_R)\\
&\leq \P(\exists I^n_i\in\Ii^n, ~ \exists r=1,\dots, n^4 ~:~\tau_{n^3}(a_r\, ; n,i)<T_R)\\
&\leq \sum_{i\in\Ii^n}\sum_{r=1}^{n^3} \P(\tau_{n^4}(a_r\,;n,i)<T_R)\\
& \leq 2^{n+1}n^3 e^{-c\frac{n^4}{n}}.
\end{align*}
Since this is the general term of a convergent series, the result follows by Borel-Cantelli lemma.
\end{proof}
~\\
Now we are able to prove Proposition \ref{Prop_Kaufman_0}.
\begin{proof}[Proof of Proposition \ref{Prop_Kaufman_0}]
First, we should note that since, for every $n,i$, the sequence of stopping times $(\tau'_k(n,i))_k$ may miss some of the zeros in $p^{-1}(I^n_i)$ which could be located between $\tau'_k(n,i)$ and $\tau'_k(n,i)+2^{-n}$. But since in the worst case we would only miss half of the good squares, the conclusion of Proposition \ref{Prop_Kaufman_0} still holds.\\
Since Lemma \ref{lemme_thm0_2} holds for any $R>1$ and since $W$ as a continuous function is bounded on $[1,2]^2$, it is immediate to conclude that for a certain (random) rank $n_0$, for any $n\geq n_0$ and any $I^n_i$, there is at most $n^7$ good squares $S_{ij}^n$ in $p^{-1}(I^n_i)$.
\end{proof}

\subsection{Proof of Theorem \ref{thm0}}
In the previous subsection, we proved our "keystone" Proposition \ref{Prop_Kaufman_0} Now we still have to conclude our proof of Theorem \ref{thm0}. We will proceed in three steps. In Step 1, as we already said it, we will prove \eqref{thm0_eq} for a fixed linear map $p~:~\R^2\to\R$ such that $\ker(p)\cap\dot{\R}^2_+\neq\{0\}$, \textit{ie} such that the zeros of $p$ are the graph of an increasing linear function. Then, in Step 2, we will generalize \eqref{thm0_eq} for any linear map $p~:~\R^2\to\R$. Finally, in Step 3, we will prove Theorem \ref{thm0}, \textit{ie} we will show that \eqref{thm0_eq} holds simultaneously for all linear map $p~:~\R^2\to\R$.

\paragraph{Step 1}
Let's fix a linear map $p~:~\R^2\to\R$ such that $\ker(p)\cap\dot{\R}^2_+\neq\{0\}$. Obviously we have $dim_H[p(\Zz)] \leq \dim_H(\Zz)=2-\frac{2}{2}=1$ by Proposition \ref{Ehm}. The point is to show that the converse inequality is also true.\\
We fix $\alpha>\beta>\dim_H(p(\Zz))$, and we take a covering $\Ee$ of $p(\Zz)$ such that $\Hh^\beta(\Ee)<\infty$. Since the covering consists of arbitrary small cubes, we may suppose that all the coverings consists of good intervals of side length at least $n_0$, where $n_0$ is taken from Proposition \ref{Prop_Kaufman_0}. For $n\geq n_0$, we note by $e_n$ the number of good intervals $I^n_i$ of side length $2^{-n}$ in the covering $\mathcal{E}$. Thus, by Proposition \ref{Prop_Kaufman_0}, a covering of $\Zz$ is given by the reunion of at most $n^7$ good squares above each of the good interval $I^n_i$ of $\mathcal{E}$. Noting that those squares are tilted, we have to consider non-tilted squares of side length $2^{-n+1}$ instead of $2^{-n}$. This way, we have that the Hausdorff measure of $\Zz$ checks
\begin{align*}
\Hh^\alpha(\Zz) \leq \sum_{n=n_0}^\infty 2^{-\alpha(n-1)}\,n^7e_n =  2^\beta \sum_{n=n_0}^\infty\,2^{-(\alpha-\beta)(n-1)}\,n^7.2^{-\beta n}e_n
\end{align*}
Thus, if we suppose $n_0$ big enough for $2^\beta\,2^{-(\alpha-\beta)(n-1)}\,n^7<1$ for all $n\geq n_0$, we get that $\Hh^\alpha(\Zz)<\infty$ and we can conclude that \eqref{thm0_eq} is indeed almost surely true.

\paragraph{Step 2} Now, if the linear map $p~:~\R^2\to\R$ we are taking is not such that $\ker(p)\cap\dot{\R}^2_+\neq\{0\}$ (\textit{ie} not such that the zeros of $p$ are the graph of an increasing linear function). Then for every $x\in\R$, $p^{-1}(x) = \{c_1(p).x-c_2(p).t~:~t\in\R\}$, where $c_1$ and $c_2$ are constants only depending on $p$. Noting $\tilde{\Zz}:=\{(s,t)\in[1,2]\times (\frac{1}{2},1]~:~tW(s,\frac{1}{t}) = 0)\}=\{(s,\frac{1}{t})\in[1,2]^2~:~W(s,t)=0\}$, our objective in this step is to prove that, almost surely $\dim_H(p(\tilde{\Zz}))=\dim_H(\tilde{\Zz})$. This would prove our point, since we have $\dim_H(\Zz)=\dim_H(\tilde{\Zz})$ and $\dim_H(p(\Zz))=\dim_H(p(\tilde{\Zz}))$. To do so, we note that the process $(s,t)\mapsto tW(s,\frac{1}{t})$ is a $(N,d)$-Brownian sheet. Thus, we can adapt the proof of Theorem \ref{thm0} to this process. The only notable difference is that in this setting, the tubes $p^{-1}(I^n_i)$ will be curved tubes whose middle line are of the form $\{\frac{1}{c_1(p).x-c_2(p).t}~:~t\in\R\}$ for $x$ the center of $I^n_i$. Since $c_2(p) >0$ by assumption on $p$, the process along this line is still a Brownian motion up to a smooth renormalization of time, as in Remark \ref{remark_renormalization}. The rest of the proof follows as in first step.

\paragraph{Step 3} Next we have to prove Theorem \ref{thm0}, \textit{ie} we have to show that \eqref{thm0_eq} holds almost surely simultaneously for every linear map $p$. We will describe the linear map $p_\theta$ we are speaking about by the angle $\theta$ its kernel forms with the $y=0$ line.\\
~\\
First, define for every $m\in\N^\times$, $\theta_0=0$, and $\theta_j^m=2\pi\frac{j}{2^m}$, where $j=1,\dots,2^m$, and notice that for any $\theta\in [0,2\pi)$, and for any $m\in\N^\times$, there exists $j=j(\theta,m)\in \{1,\dots,2^m\}$ such that
\begin{equation}
\label{ang1}
|\theta^m_j-\theta|\leq \frac{2\pi}{2^m}.
\end{equation}
Secondly, an easy adaptation of the proof of Lemma \ref{lemme_thm0_2} shows that for any $R>1$, almost surely, there exists a (random) rank $n_0$ such that for all $n \geq n_0$,
\begin{equation}
\label{ang2}
\forall j=1,\dots,2^n, \qquad \forall I^n_i\in\Ii^n(\theta^n_j)~~:~~\tau_{n^7}'(n,i\,;\,\theta^n_j)\geq T_R.
\end{equation}
Thirdly, fix $\theta\in[0,2\pi)$ and $n\in\N^\times$, and consider any of the intervals $I^n_i$ associated to the linear map $p_\theta$. Remember that we want to control the number of good squares above the interval $I^n_i$. To do so, one can show that with $\theta^n_j$ given by \eqref{ang1}, there exists a constant $c_3>1$ independent of $\theta$, $n$ and $i$, such that, for a certain interval $I'$ of length $c_3.|I^n_i|$ where $|I^n_i|$ is the length of $I^n_i$, the tube $p_{\theta^n_j}^{-1}(I')$ overlaps entirely the tube $p_\theta^{-1}(I^n_i)$ (one should remember that we are working only on $[1,2]^2$). Moreover one can show that this $c_3$ is less than $10$. Thus, using \eqref{ang2}, we have that for any $R>1$, almost surely, there exists a (random) rank $n_0$ such that for all $n \geq n_0$,
\begin{equation}
\label{ang3}
\forall \theta\in[0,2\pi) \qquad \forall I^n_i\in\Ii^n(\theta)~~:~~\tau_{10n^7}'(n,i\,;\,\theta)\geq T_R.
\end{equation}
From \eqref{ang3}, one can conclude as in the proof of Proposition \ref{Prop_Kaufman_0} that there exists a (random) rank $n_0$ such that, for any $\theta\in [0,2\pi)$, any $n\geq n_0$ and any $i$, there is at most $10n^7$ good squares $S_{ij}^n$ in $p_\theta^{-1}(I^n_i)$. One can then end the proof of Theorem \ref{thm0} as in Step 1.

\newpage

\section{Proof of Theorem \ref{thm1}}
\label{section_thm1}
This section will be dedicated to the proof of Theorem \ref{thm1}.\\
To compute the Hausdorff dimension of a given set $A$ of a space $\R^n$, it is enough to compute the Hausdorff dimension of $A\cap C$ for every bounded hypercube of $\R^n$. Thus, for simplicity reasons we will only work on the square $[1,2]^N$ of $\R^N$, \textit{ie} we will consider that the Brownian sheet is only defined on $[1,2]^N$, unless the opposite is explicitly stated. The case of the other squares of the form $[a,b)^N$ for $0<a<b$ can be done the in same way. Additionally, we can suppose without loss of generality throughout this section that $\Zz\cap[1,2]^N\neq\emptyset$, since we can always take compact subset of $\R_+^N$ big enough to intersect $\Zz$. Finally, one should note that the problems of definitions that may occur at the boundary of $[1,2]^N$ can be ignored thanks to the fact that our proof works for any bounded cube $C$.

\paragraph{Notations}~\\
The notation $|\cdot|$ will always denote the usual euclidean norm for the space it will be defined on, and $B(a,r)$ will always denote the associated closed ball centered in $a$ and of radius $r$. Similarly, $|\cdot|_\infty$ will denote the usual $\sup$-norm on the space it will be defined on, and $B^\infty(a,r)$ will always denote the associated open ball centered in $a$ and of radius $r$.\\
Finally, the notations $C$ and $C(\cdot)$ will denote throughout this section a constant that does not depend on any relevant variable, and may change its value from one line to the other.\\
~\\
As in the proof of Theorem \ref{thm0}, we would like to prove an intermediate result analogous to Proposition \ref{Prop_Kaufman_0}. Unfortunately, the strategy followed in the proof of Proposition \ref{Prop_Kaufman_0} can't be generalized to higher dimensions because there are no well-suited equivalents for the stopping times $\tau_k(a;n,i)$ from the proof of Theorem \ref{thm0} in higher dimensions. For this reason, we have to change the approach we take, and the adaptation led us to Lemma \ref{lemma_thm1}. It turns out that proving Lemma \ref{lemma_thm1} enables us to prove directly the unpublished Proposition \ref{prop_kaufman_1}. This is why we chose to prove first Proposition \ref{prop_kaufman_1} in the next subsection with the help of Lemma \ref{lemma_thm1}, before turning back to the demonstration of Theorem \ref{thm1}.

\begin{proposition}
\label{prop_kaufman_1}
For every $N \leq \frac{d}{2}$, almost surely for every linear subspace $K\subset\R^N$ of dimension $N'':=\dim(K)\leq\frac{d}{2}$ and for every Borel set $E\subset K$, if $K\cap [1,2]^N \neq \emptyset$, then 
\begin{align}
\label{prop_kaufman_1_eq}
    2\dim_H\{E\}=\dim_H\{W(E)\}.
\end{align}
\end{proposition}

\noindent An immediate corollary of Proposition \ref{prop_kaufman_1} is the following:\\
Let $B: \R^N_+\rightarrow \R^d$ be a Brownian sheet with $N \leq \frac{d}{2}$, then for every Borel set $E\subset \R^N$,
\begin{align}
\label{prop_kaufman_1_cor}
    2\dim_H\{E\}=\dim_H\{W(E)\} \ as .
\end{align}
This result generalizes a well-known result of Kaufman in \cite{kaufman1969propriete} for $N=1$ and Mountford in \cite{mountford1989uniform} for N=2. 

\begin{remark}
\label{remark_parametrization}
To locate ourselves in the subspaces $K$, we will need a parametrization of $K$. Note $(e_i)_{1\leq i \leq N}$ the canonical basis of $\R^N$.\\
Let $(k_i)_{i\leq \binom{N}{N''}}$ be the family of subsets of size $N''$ in $\{1,\dots,N\}$ and $V_i:=\text{Vect}(e_j ~:~j\in k_i)$.
In this case, let $q_i$ be the orthogonal projection on $V_i$ and $\Gamma_i ~:~V\to K$ be the parametrization such that $\Gamma_i(x)-x \perp V_i$. We have $q_i \circ \Gamma_i(x) = x$ for every $x \in V_i$. In the following sections, we will choose the $V_i$ given by Lemma \ref{lemma_parametrization}, such that we can drop the subscript $i$ for $V,\,q$ and $\Gamma$, but one should note that those still depend on $K$. Finally, we can define $\tilde{W}~:~V\to \R^d$ by $\tilde{W}=W\circ \Gamma$.\\
For the simplicity of the notations, we will suppose in the rest of this proof that $V=\text{Vect}(e_1,\dots,e_{N''})$ as elements of $\R_+^{N''}$.\\
\end{remark}
\noindent The reason why we chose this parametrization in particular is twofold. On the first hand we need to use the result in \cite[eq (4.5)]{khoshnevisan.xiao} to prove the inequality \eqref{ineq_kosh.xiao}. On the second hand we need a parametrization which does not change the fractal properties. More precisely, with this parametrization it is easy by Lemma \ref{lemma_parametrization} to see that $\dim_H(W^{-1}(0)) = \dim_H(\tilde{W}^{-1}(0))$ and $\dim_H(p(W^{-1}(0))) = \dim_H(p(\tilde{W}^{-1}(0)))$. Thus, in Theorem \ref{thm1} instead of proving \eqref{thm1_eq}, we can prove that 
\begin{align} 
\label{thm1_eq_tilde}
\dim_H(\tilde{\Zz}) = \dim_H(p(\tilde{\Zz})),
\end{align}
where $\tilde{\Zz} = \tilde{W}^{-1}(0)$. Similarly, instead of proving \eqref{prop_kaufman_1_eq}, we can prove that 
\begin{align}
\label{prop_kaufman_1_eq_tilde}
    2\dim_H\{E\}=\dim_H\{\tilde{W}(E)\},
\end{align}
where $E\subset V$.

\subsection{Proof of Proposition \ref{prop_kaufman_1}}
In a first time, and until otherwise explicitly stated, we will do the proof of Proposition \ref{prop_kaufman_1} for a fixed linear subspace $K\subset\R^N$.\\
With Corollary \ref{cor_parametrization}, one can easily prove that for any Borelian $E\subset K$, 
\begin{align}
\label{ineq_prop_kaufman_1}
    \dim_H(\tilde{W}(E))\leq 2\dim_H(E).
\end{align}
Thus, our objective here is only to prove that for any Borelian $E\subset K$,
\begin{align}
\label{ineq_prop_kaufman_2}
   \dim_H(\tilde{W}(E))\geq 2\dim_H(E).
\end{align}
Since $E \subset \tilde{W}^{-1}(\tilde{W}(E))$, we can suppose without loss of generality that $E = \tilde{W}^{-1}(\tilde{W}(E))$.
Our objective until the end of this subsection is now to prove \eqref{ineq_prop_kaufman_2}. We will begin with a Lemma, which will be playing the role that Proposition \ref{Prop_Kaufman_0} had in the demonstration of Theorem \ref{thm0} in our generalized case.

\begin{lemma} 
\label{lemma_thm1}
Fix $R>0$.\\
There exists a universal constant $k>d$ and a (random) rank $n_0$ such that, for any $n\geq n_0$ and any ${a\in[-R,R]^d}$, the set $\tilde{W}^{-1}(B^\infty(a,2^{-\frac{n}{2}}))$ can be covered by at most $n^k$ dyadic cubes of side length $2^{-n}$.
\end{lemma}

\begin{proof}
We will proceed in two steps. In a first step, we will prove Lemma \ref{lemma_thm1} for a fixed $a\in[-R,R]^d$. In a second step we will show how to use this intermediate result to find a rank $n_0$ uniform in $a\in[-R,R]^d$.\\
For the sake of simplicity, we will identify in this proof the subspace $V$ with $\R^{N''}$, and we have that $[1,2]^{N''}$ parametrizes the set $K\cap [1,2]^N$. We define, for any integer $n\geq 1$, $\mathcal{D}^n:=(D_i^n)_{i=1,\dots,2^{nN''}}$ a disjoint covering of $[1,2]^{N''}$ consisting of the squares of the form $\prod_{j=1}^{N''} [1+\tfrac{k_j}{2^n},1+\tfrac{k_j+1}{2^n})$ where the $k_j$'s are integers between $0$ and $2^n-1$, and we will denote by $t_i^n$ the center of the square $D_i^n$.\\
For $a\in[-R,R]^d$, we introduce the events
$$B_i^n(a) := \{\exists t \in D^n_i ~:~|\tilde{W}_t-a|_\infty \leq 2^{-\frac{n}{2}}\} \qquad \text{and} \qquad A_i^n(a):=\{|\tilde{W}_{t_i^n}-a|\leq n2^{-\frac{n}{2}}\}.$$
With the modulus of continuity introduced in Corollary \ref{cor_parametrization}, if $n\geq n_0$ a constant independent of $K$, we have the inclusion $B_i^n(a)\subset A_i^n(a)$ \textit{as}. Thus defining $Y^n(a):=\sum_{i=1}^{2^{nN''}}\1_{B_i^n(a)}$ and $X^n(a):=\sum_{i=1}^{2^{nN''}}\1_{A_i^n(a)}$ we also have ${Y^n(a)<X^n(a)}$ \textit{as}.

\paragraph{Step 1 :} 
We fix $a\in[-R,R]^d$ and we note $B:=B^\infty(a,2^{-\frac{n}{2}})$. For the simplicity of notations, we drop the dependencies in $a$ until the end of this Step.\\
Our first objective is to control the quantities $\E[(X^n)^{r+1}]$ for any integer $r\geq 1$. Thus, we compute with the conditional tower property, for any integer $r\geq 1$,
\begin{align*}
	\E[(X^n)^{r+1}]
	&=\sum_{q_1, q_2,\dots, q_r =1}^{2^{nN''}}\,\sum_{q_{r+1}=1}^{2^{nN''}}\E[\1_{A^n_{q_1}}\dots\1_{A^n_{q_r}}\cdot\1_{A^n_{q_{r+1}}}]\\
	&=\sum_{q_1, q_2,\dots, q_r=1}^{2^{nN''}}\P(A^n_{q_1} \cap \dots \cap A^n_{q_r})\cdot\sum_{q_{r+1}=1}^{2^{nN''}}\P(A^n_{q_{r+1}}~|~A^n_{q_1},\dots,A^n_{q_r}).
	\end{align*}

We fix $1 \leq q_1,\dots,q_r \leq 2^{nN''}$ and we define $T:=\{t^n_1,\dots,t^n_{2^{nN''}}\}$ and $T_q:=\{t^n_{q_1},\dots,t^n_{q_r}\}$. Then, for $x\geq 0$ we also define
\begin{align}
\label{def_tx}
    T_x := \{ t \in T-T_q~:~\frac{x}{2} < \max_{i=1,\dots,N''}\min_{s\in T_q}|t^i-s^i| \leq x \}.
\end{align}
One should note that $(T_{2^{-n+j}})_{j=0,\dots,n} \cup T_0$ is a partition of $T-T_q$.\\
We will need the following lemma:
\begin{lemma}
\label{lemma_6}
We have for $t\in T_x$,
$$\P(|\tilde{W}_t-a|<n2^{-\frac{n}{2}} ~|~ A^n_{q_1},\dots,A^n_{q_r}) \leq 2^d\cdot \frac{n^d\cdot2^{-\frac{n}{2}d}}{x^{\frac{d}{2}}}.$$
\end{lemma}

\begin{proof}
First, we notice that,
\begin{align*}
    \sigma^2 &:= Var(\tilde{W}_t~|~\tilde{W}_{t_{q_1}},\dots,\tilde{W}_{t_{q_r}}) = Var(W\circ \Gamma(t)~|~W\circ \Gamma(t_{q_1}),\dots,W\circ \Gamma(t_{q_r}))
\end{align*}
By the result in \cite[eq (4.5)]{khoshnevisan.xiao}, we deduce
\begin{align}
\label{ineq_kosh.xiao}
    \sigma^2 \geq \dfrac{1}{2}\sum_{k=1}^{N} \min_{j=1,\dots,r} |\Gamma(t)^k-\Gamma(t_{q_j})^k| \geq \dfrac{1}{2}\sum_{k=1}^{N''} \min_{j=1,\dots,r} |t^k-t^k_{q_j}| \geq \frac{1}{2} \max_{i=1,\dots,N-N'}\min_{s \in T_q} |t^i-s^i| \geq \frac{x}{4}. 
\end{align} 
Thus, we compute, knowing that $W$ is a Gaussian process, 
\begin{align*}
    \P(|\tilde{W}_t-a|<n2^{-\frac{n}{2}} ~|~ A^n_{q_1},\dots,A^n_{q_r}) &\leq \left(\dfrac{1}{2\pi \sigma^2}\right)^{\frac{d}{2}}\prod_{i=1}^d \int_{a-n2^{-\frac{n}{2}}}^{a+n2^{-\frac{n}{2}}} e^{-\frac{(y-a)^2}{2\sigma^2}}dy\\
    &\leq x^{-\frac{d}{2}}(2n2^{-\frac{n}{2}})^d \\
    &\leq 2^d\cdot\frac{n^d\cdot2^{-\frac{n}{2}d}}{x^{\frac{d}{2}}}.
    \end{align*}
\end{proof}

\noindent Now, we can compute, using Lemma \ref{lemme_tx} and Lemma \ref{lemma_6},
\begin{align*}
    \sum_{q_{r+1}=1}^{2^{nN''}} \P(A^n_{q_{r+1}} ~|~ A^n_{q_1},\dots,A^n_{q_r}) &= \sum_{q_{r+1}\in T-T_q} \P(A^n_{q_{r+1}} ~|~ A^n_{q_1},\dots,A^n_{q_r}) + \sum_{q_{r+1}\in T_q} \P(A^n_{q_{r+1}} ~|~ A^n_{q_1},\dots,A^n_{q_r})\\
    &=\sum_{j=0}^n \sum_{t\in T_{2^{-n+j}}} \P(|\tilde{W}_t-a|<n2^{-\frac{n}{2}} ~|~ A^n_{q_1},\dots,A^n_{q_r})\\
    &+ \sum_{t\in T_0}\P(|\tilde{W}_t-a|<n2^{-\frac{n}{2}} ~|~ A^n_{q_1},\dots,A^n_{q_r}) + \sum_{q_{r+1}\in T_q} 1\\
    &\leq \sum_{j=0}^n  (r2^{n+1}2^{-n+j})^{N''} ~ \dfrac{Cn^d2^{-\frac{n}{2}d}}{(2^{-n+j})^{\frac{d}{2}}} +(2r)^{N''} \\
    &\leq r^{N''}2^{N''}Cn^d \sum_{j=0}^n 2^{jN''}2^{-j\frac{d}{2}} +r \\
    &\leq Cr^{N''}n^d.
\end{align*}
Thus, we have (with an iteration in last step)
\begin{align*}
    \E[(X^n)^{r+1}] &\leq Crn^d2^{n(N''-\frac{d}{2})} \sum_{q_1, q_2,\dots, q_r=1}^{2^{nN''}}\P(A^n_{q_1} \cap \dots \cap A^n_{q_r})\\
    &= Crn^d2^{n(N''-\frac{d}{2})} \E[(X^n)^r]\\
    & \leq C^r (r!)^{N''} n^{dr} \E[X^n].
\end{align*}
But, we have with Lemma \ref{lemma_6},
\begin{align*}
    \E[X^n]&=\sum_{t\in T} \P(|\tilde{W}_t-a|<n2^{-\frac{n}{2}}~|~\tilde{W}_0)\\
    &\leq \sum_{t\in T} C\cdot\frac{n^d\cdot2^{-\frac{n}{2}d}}{(|t|_\infty)^{\frac{d}{2}}}\\
    &\leq \sum_{t\in T} C\cdot\frac{n^d\cdot2^{-\frac{n}{2}d}}{(2^{-n})^{\frac{d}{2}}}\\
    &\leq 2^{nN''}Cn^d.
\end{align*}
Thus, finally using Markov inequality and reuniting all the bricks we get, taking $r=n$ and a $k$ that will be chosen later,
\begin{align*}
    \P(X^n \geq n^k) &\leq \dfrac{\E[(X^n)^{n+1}]}{n^{kr}}\\
    &\leq n^{-k(n+1)}.C^{n+1}((n+1)!)^{N''} n^{d(n+1)}Cn^d \\
    &= C^{n+2}((n+1)!)^{N''}n^{(d-k)(n+1)+d}2^{nN''}
\end{align*}
We get finally
\begin{align}
\label{thm1_preuve_a_fixe}
    \P(X^n\geq n^k) \leq C (n!)^{N''} 2^{nN''} n^{n(d-k)},
\end{align}
which is the general term of a converging series if $k>d$. Thus, we can apply the Borel-Cantelli Lemma to prove the existence of a (random) rank $n_0$ such that \textit{as} for every $n\geq n_0$, $X^n<n^k$. Finally, since $Y^n \leq X^n$ \textit{as}, we have the desired result for a fixed $a\in[-R,R]^d$.

\paragraph{Step 2 :}
Since our objective is to find a rank $n_0$ independent of $a\in[-R,R]^d$, this control for a fixed $a$ is not enough, and we have to recover the dependency in $a$.\\
The trick is to take any finite covering $\tilde{\mathcal{D}}^n:=(B^\infty(a_i^n,2^{-\frac{n}{2}}))_{i\in \mathcal{I}^n}$ of the hypercube $[-R,R]^d$. For the sake of simplicity consider that the $a_i^n$ are just the points of $n$-dyadic coordinates (\textit{ie} of the form $\frac{k}{2^{\frac{n}{2}}}$). In this case, if we find a rank $n_0$ uniform in $a_i^n$ with $i\in\mathcal{I}^n$, since every hypercube $B^\infty(a,2^{-\frac{n}{2}})$ can be covered by at most $2^d$ hyper-cubes from $\tilde{\mathcal{D}}^n$, the result would still hold.\\
We can now compute the sum, noting that the set $\mathcal{I}^n$ contains $(R2^{\frac{n}{2}+1})^d$ elements,
\begin{align*}
    \sum_{n=1}^\infty \P(\exists i\in\mathcal{I}^n : X^n(a_i^n) > n^k) & \leq \sum_{n=1}^\infty \sum_{i=1}^{(R2^{\frac{n}{2}+1})^d}\P(X^n(a_i^n) > n^k)\\
    &\leq \sum_{n=1}^\infty (R2^{\frac{n}{2}+1})^d\P(X^n(a_i^n) > n^k).
\end{align*}
Then, by \eqref{thm1_preuve_a_fixe} we get
\begin{align*}
    \sum_{n=1}^\infty \P(\exists i\in\mathcal{I}^n : X^n(a_i^n) > n^k) & \leq \sum_{n=1}^\infty (R2^{\frac{n}{2}+1})^dC (n!)^{N-N'} n^{n(d-k)}
\end{align*}
which is the general term of the convergent series if and only $k>d$. Thus, applying Borel-Cantelli Lemma as in Step 1, one can deduce that there exist a (random) rank $n_0$ such that for every $n\geq n_0$ and every $i\in\mathcal{I}^n$ the desired conclusion holds.
\end{proof}
~\\
Now we can finish this subsection with the proof of \eqref{ineq_prop_kaufman_2}. We note $F:=\tilde{W}(E)$.\\
Since $F=\cup_{R\in \N} F\cap[-R,R]^d$, we can suppose without loss of generality that $F\subset[-R,R]^d$.\\
Let $\beta>\dim_H(F)$. There exists a covering $\mathcal{C}$ of $F$ which can be assumed to consist of sets of the form $B^\infty(a,2^{-\frac{n}{2}})$ for $n\geq n_0$ given by Lemma \ref{lemma_thm1} and such that $\mathcal{H}^\beta(\mathcal{C})<\infty$. Thus, noting $N^n(\mathcal{C})$ the number of hyper-cubes of side length $2^{-\frac{n}{2}}$ in this covering, we have that
\begin{align}
    \label{thm1_preuve_eq1}
    \sum_{n\geq n_0} N^n(\mathcal{C}) 2^{-\frac{n}{2}\beta}<\infty.
\end{align}
From Lemma \ref{lemma_thm1} we have that each $\tilde{W}^{-1}(B^\infty(a,2^{-\frac{n}{2}})) \subset \mathcal{C}$ can be covered by at most $n^k$ cubes in $\mathcal{D}^n$ of side length of the form $C2^{-n}$, which are forming a covering of $\tilde{W}^{-1}(F)$. Thus, for any $\alpha>\beta$, we have that
$$ \mathcal{H}^{\frac{\alpha}{2}}(\tilde{W}^{-1}(F)) \leq \sum_{n\geq n_0} N^n(\mathcal{C})n^k2^{-n\frac{\alpha}{2}}<\infty.$$
From the arbitrariness of $\beta$ and $\alpha$, we can conclude that $2\dim_H(E)\leq \dim_H(\tilde{W}(E))$, which concludes the proof for a fixed linear subspace $K\subset\R^N$.\\
~\\
From now on, we will explain why does Proposition \ref{prop_kaufman_1} indeed apply for every linear subspace $K\subset\R^N$ simultaneously. We will follow a path similar to Step 3 in the proof of Theorem \ref{thm0}.\\
Firstly, one should note that for every $m\in\N^\times$, one can find a family $\mathcal{K}^m$ of linear sub-spaces $(K_j^m)_{j\leq \#\mathcal{K}^m}$ of $\R^N$ such that $\#\mathcal{K}^m \leq C2^{nN}$ and for every linear subspace $K\in\R^N$ 
\begin{equation}
    \label{angk1}
    \min_{j\leq \#\mathcal{K}^m_j} d(K_j^m,K)<2^{-m},
\end{equation}
where $d$ is the Hausdorff distance between to subsets of $\R^N$.\\
Secondly, since $\#\mathcal{K}^m \leq C2^{nN}$, an easy adaptation of the proof of Lemma \ref{lemma_thm1} shows that almost surely,
\begin{align}
    \label{angk2}
    \sum_{n=1}^\infty \P(\exists K^n_j \in \mathcal{K}^m , \exists i\in \mathcal{I}^n ~:~X^n(a^n_i)>n^k) < \infty,
\end{align}
where the sets $\mathcal{I}^n$, $X^n$ and $a^n_i$ are of course defined with respect to the linear subspace $K^n_j$. From \eqref{angk2} we can deduce that for every $R>0$, there exists a (random) rank $n_0$ such that for every $n\geq n_0$ every $a\in[-R,R]^d$ and every $K_j^n\in\mathcal{K}^n$, the set $\tilde{W}_j^{-1}(B^\infty(a,n2^{-n/2}))$, where $\tilde{W}_j$ is of course defined with respect to the linear subspace $K_j^m$, can be covered by at most $n^k$ dyadic cubes of side length $2^{-n}$.\\
Thirdly, fix $K$ any linear subspace of $\R^N$, $n\in\N^\times$ and $a\in[-R,R]^d$. Note that by the modulus of continuity of $\tilde{W}$ showed in Corollary \ref{cor_parametrization} we have the inclusion
\begin{align}
    \label{angk3}
    \tilde{W}^{-1}(B^\infty(a,2^{-n/2})) ~\subset ~ \tilde{W}^{-1}(B^\infty(a,n2^{-n/2})) ~ \subset ~ \tilde{W}_j^{-1}(B^\infty(a,2n2^{-n/2})) ,
\end{align}
where we recall that $\tilde{W}=W\circ \Gamma_K$ and $\tilde{W}_j=W\circ \Gamma_{K^n_j}$ for $K_j^n$ being the indices for which the $\min$ in \eqref{angk1} is attained. Moreover, since $B^\infty(a,2n2^{-n/2}) \subset \cup_{l=1}^{2^d} B^\infty(a_l^n,n2^{-n/2})$ for certain points $a_l^n\in[-R,R]^d$, we have
\begin{align}
    \label{angk4}
    \tilde{W}^{-1}(B^\infty(a,2^{-n/2})) ~ \subset ~ \bigcup_{l=1}^{2^d} \tilde{W}_j^{-1}(B^\infty(a_l^n,2n2^{-n/2})).
\end{align}
Thus, using the fact that for $n\geq n_0$ the set $\tilde{W}_j^{-1}(B^\infty(a,n2^{-n/2}))$ can be covered by at most $n^k$ dyadic cubes of side length $2^{-n}$, we showed that there exists a (random) rank $n_0$ such that the set $\tilde{W}^{-1}(B^\infty(a,2^{-n/2}))$ can be covered by at most $n^k$ dyadic cubes of side length $2^{-n}$.\\
Finally, one can conclude as in the first part of the proof of Proposition \ref{prop_kaufman_1}, in which we were reasoning for a fixed linear subspace $K\in\R^N$.

\subsection{Proof of Theorem \ref{thm1}}
As in the proof of Theorem \ref{thm0}, we can use Lemma \ref{lemma_thm1} with $a=0$ to show Theorem \ref{thm1} for a fixed projection $p$. Then we can adapt the end of the proof of Proposition \ref{prop_kaufman_1} to deduce that the result holds simultaneously for every projection $p$.

\newpage
\section{Proof of Theorem \ref{thm2}}
\label{section_thm2}
This section will be dedicated to the proof of Theorem \ref{thm2}.
\paragraph{Notations}
We note $(e_i)_{1\leq i \leq N}$ the canonical basis of $\R^N$. The notations $|\cdot|$ and $\lambda(\cdot)$ will always denote the usual euclidean norm and the usual Lebesgue measure for the space it is defined on, and $B(x,r)$ is the associated open ball of center $x$ and radius $r>0$. Similarly, $|\cdot|_\infty$ will denote the usual $\sup$-norm on the space it is defined on. The number of elements of a set $Q$ is noted $\# Q$.\\
We will reserve the notation $p$ for orthogonal projections $p: \R^N\rightarrow \R^{N'}$ with $N'<N-\frac{d}{2}$ where $N':=rank(p)$. The variable $s\in \R^{N'}$ will denote a point in the projective space of $p$. Since the sets $p^{-1}(s)$ are sub-spaces of $\R^N$, we can use Remark \ref{remark_parametrization} to find a subset $k\subset \{1,\dots,N\}$ of size $N'':=N-N'$ and an orthogonal projection $q~:~ p^{-1}(s) \mapsto \text{Vect}(e_j~:~j\in k)$ such that its inverse function is $N$-lipschitz. We will note $V(p):=\text{Vect}(e_j~:~j\in k)$ which can be shown does not depend on the point $s$ we considered, the orthogonal projection $q(s,p)$ does depend also on the point $s$ and its inverse $\Gamma(s,p)$ will be called the parametrization function. As previously we will also define $\Bar{W}(s,p):=W\circ \Gamma(s,p)$.\\
Moreover, we will often not work on all $\R^N$ but rather on compact rectangles $M\subset \R^N$ and we will study the local time of the process $\Bar{W}(s,p):=W\circ \Gamma(s,p)$ considering that $W$ is restricted to $M$. For any $x\in \R^d$ we will note this local time spent by $\Bar{W}$ around the value $x$ by $L_M(x\,;s,p)$.\\
In the cases where it does not lead to any confusion, we may allow ourselves to drop the dependencies in $s$, $p$, $x$ or in $M$.\\
Finally, the notations $C$ and $C(\cdot)$ will denote throughout this section a constant that does not depend on any relevant variable, and may change its value from one line to the other.\\

\noindent We will proceed in three steps. In a first time in Proposition \ref{thm2_1} we prove that the local time is continuous in probability for the space variable. In a second time in Proposition \ref{thm2_2} we use Proposition \ref{thm2_1} to prove that the local time is continuous jointly the parameters $s$ and $p$. One has to notice that the space where $p$ lives is homeomorphic to a subspace of $\R^{N.N'}$, such that we can avoid to specify the exact norm we will use in the definition of continuity, since they are all the same up to an irrelevant constant. Finally in Section \ref{section_thm2_final} we will prove why Proposition \ref{thm2_2} implies Theorem \ref{thm2}.

\begin{proposition}
    \label{thm2_1}
    Fix a projection $p$ and a point $s$ and let $\Bar{W}$ be the associated process. Let $M \subset \R_+^N$ be a bounded rectangle. For any $0<\alpha<\frac{1}{2}\min(1,N''-\frac{d}{2})$ there exists $c>0$ such that for $x,y \in \R^d$ close enough
    $$ \P(|L_M(x\,;s,p)- L_M(y\,;s,p)| \geq |x-y|^\alpha) \leq e^{-c|x-y|^{-\alpha}}. $$
\end{proposition}

\begin{proposition}
    \label{thm2_2}
    For any fixed $x\in \R^d$ and rectangle $M$, almost surely the application $(s,p) \longmapsto L_M(x\,;s,p)$ is continuous.
\end{proposition}

\subsection{Proof of Proposition \ref{thm2_1}}
Our approach is inspired by the work in \cite{ehm1981sample}. It relies on the following lemma.
\begin{lemma}
\label{thm2_1_lemma1}
Fix a projection $p$ and a point $s$ and let $\Bar{W}$ be the associated process. Let $\bar{M}\subset V$ be a bounded rectangle. Define for $m\in \N$ and $ 0\leq \eta < N''-\frac{d}{2}$,
	$$	J_{\bar{M}}(2m,\eta) := \int_{\bar{M}^{2m}}\int_{(\R^d)^{2m}}\left| \E e^{i\sum_{j=1}^{2m}\beta_j\cdot \tilde{W}(t^j)} \right|\prod_{j=1}^{2m}|\beta_j|^\eta d\beta_j dt_j.$$	
Then, for any $ 0\leq \eta < N''-\frac{d}{2}$ there exists a finite constant $C(\eta)>0$ such that, for any $m\in \N$,
	$$J_{\bar{M}}(2m,\eta) \leq C^{2m}\left(\frac{(2m)!}{g[1+2m(1-(d+\eta)/2N'')]}\right)^{N''},$$
where $g$ is the Euler Gamma-function.
\end{lemma}

\begin{proof}
For notational and computational simplicity, we will suppose that $V=\text{Vect}(e_1,\dots,e_{N''})$ and that $\bar{M}=[1,2]^{N''}$. In this case, we can decompose the process $\Bar{W}$ in a similar manner to what has been done to obtain \cite[eq (4.5)]{khoshnevisan.xiao} with, for any $s=(s^1,\dots,s^{N''})$ and $t=(t^1,\dots,t^{N''}) \in V$,
$$ \Bar{W}(t+s) = \bar{W}(t) + \sum_{k=1}^{N''} X_k(s^k\,;t) + R(s\,;t), $$
where the $X_k$ are one-dimensional martingale processes and $R(s\,;t) = \int_{[t,t+s]}d\bar{W}(r)$. Moreover, the variables $\bar{W}(t)$, $X_k(s^k\,;t)$ and $R(s\,;t)$ are independent. With this independence, we obtain for fixed $(t_j)_{j\leq 2m}$ points in $V$, noting $t:=(\min_{j\leq 2m} t_j^1,\dots,\min_{j\leq 2m} t_j^{N''})$ and $s_j := t_j - t$,
\begin{align*}
    \idotsint_{(\R^d)^{2m}}\left| E[e^{i\sum_{j=1}^{2m}\beta_j\cdot\tilde{W}(t+s_j)}] \right|\prod_{j=1}^{2m}|\beta_j|^{\eta}d\beta_j 
    & \leq \idotsint_{(\R^d)^{2m}}\left| \prod_{k=1}^{N''}\E[e^{i\sum_{j=1}^{2m}\beta_j\cdot X_k(s_j^k\,;t)}] \right|\prod_{j=1}^{2m}|\beta_j|^{\eta}d\beta_j \\
    & \leq \prod_{k=1}^{N''}\left( \idotsint_{(\R^d)^{2m}} \left|\E[e^{i\sum_{j=1}^{2m}\beta_j\cdot X_k(s_j^k\,;t)}]\right|^{N''} \prod_{j=1}^{2m}|\beta_j|^{\eta}d\beta_j \right)^{\frac{1}{N''}},
\end{align*}
where we used Hölder's inequality in the last line.\\
From the characteristic function of $W$ and the fact that $\Gamma$ is $N$-lipschitz we obtain, for fixed $\beta \in \R^d$, $k$ and $s,t \in V$,
\begin{align}
    \label{thm_2_lemme_ineq}
    \E[e^{i\beta\cdot X_k(s^k\,;t)}] \leq  e^{-N^{N''}\tfrac{s^k}{t^k}\lambda([0,t]^{N''}) |\beta|^2}. 
\end{align} 
Even if it means rewriting the sum in a different order, we can suppose without loss of generality that, for any fixed $k$, it holds $s_j^k>s^k_{j-1}$ for any $j\leq 2m$. This way, with the convention that $s^k_0 = 0$, we have 
$$\sum_{j=1}^{2m}\beta_j\cdot X_k(s_j^k\,;t) = \sum_{j=1}^{2m} \sum_{l=1}^j \beta_j\cdot \big[ X_k(s_l^k\,;t) - X_k(s^k_{l-1}\,;t)\big] = \sum_{l=1}^{2m} \sum_{j=l}^{2m} \beta_j\cdot \big[ X_k(s_l^k\,;t) - X_k(s^k_{l-1}\,;t)\big].$$
Thus, we compute, using that $X_k$ has independent and stationary increments
\begin{align*}
    I_k :&=\idotsint_{(\R^d)^{2m}} \left|\E\left[e^{i\sum_{j=1}^{2m}\beta_j\cdot X_k(s_j^k\,;t)}\right]\right|^{N''} \prod_{j=1}^{2m}|\beta_j|^{\eta}d\beta_j\\ 
    &= \idotsint_{(\R^d)^{2m}} \left|\E\left[\prod_{l=1}^{2m}e^{i\sum_{j=l}^{2m}\beta_j\cdot [X_k(s_l^k\,;t)-X_k(s_{l-1}^k\,;t)] }\right]\right|^{N''} \prod_{j=1}^{2m}|\beta_j|^{\eta}d\beta_j \\
    &= \idotsint_{(\R^d)^{2m}} \prod_{l=1}^{2m}\left|\E\left[e^{i\sum_{j=l}^{2m}\beta_j\cdot [X_k(s_l^k\,;t)-X_k(s_{l-1}^k\,;t)] }\right]\right|^{N''} \prod_{j=1}^{2m}|\beta_j|^{\eta}d\beta_j \\
    &= \idotsint_{(\R^d)^{2m}} \prod_{l=1}^{2m}\left|\E\left[e^{i\sum_{j=l}^{2m}\beta_j\cdot X_k(s_l^k-s^k_{l-1}\,;t)}\right]\right|^{N''} \prod_{j=1}^{2m}|\beta_j|^{\eta}d\beta_j \\
    &\leq \idotsint_{(\R^d)^{2m}} \prod_{l=1}^{2m} e^{-N''N^{N''}\tfrac{s^k_l-s^k_{l-1}}{t^k}\lambda([0,t]^{N''}) |\sum_{j=l}^{2m}\beta_j|^2} \prod_{j=1}^{2m}|\beta_j|^{\eta}d\beta_j\\
    &= \idotsint_{(\R^d)^{2m}} \prod_{j=1}^{2m} e^{-N''N^{N''}\tfrac{s^k_j-s^k_{j-1}}{t^k}\lambda([0,t]^{N''}) |\gamma_j|^2} |\gamma_j-\gamma_{j+1}|^{\eta}d\gamma_j.
\end{align*}
where the penultimate inequality comes from \eqref{thm_2_lemme_ineq} and the last equality is just a substitution of the variables with the notation $\gamma_j := \sum_{l=j}^{2m} \beta_l$ and the convention $\gamma_{2m+1}=0$.\\
Let $Q:=\{(q_1,\dots,q_{2m}) \in \{0,1,2\}^{2m} ~:~\sum_{j=1}^{2m} q_j = 2m \}$. It can be shown that 
\begin{align*}
\prod_{j=1}^{2m}|\gamma_j-\gamma_{j+1}|^{\eta} \leq  \prod_{j=1}^{2m}(2\max(|\gamma_j|\,;|\gamma_{j+1}|))^\eta \leq  2^{2m\eta} \sum_{q\in Q}\prod_{j=1}^{2m}|\gamma_j|^{q_j\eta},
\end{align*}
such that noting $K := \max_{q\in \{0,1,2\}} \int_{\R^d}e^{-|\gamma|^2}|\gamma|^{q\eta}d\gamma < \infty$ and using a substitution, we obtain
\begin{align*}
    I_k &\leq \sum_{q\in Q} \prod_{j=1}^{2m} 2^\eta \int_{\R^d} e^{-N''N^{N''}\tfrac{s^k_j-s^k_{j-1}}{t^k}\lambda([0,t]^{N''}) |\gamma|^2}|\gamma|^{q_j\eta}d\gamma \\
    &\leq \sum_{q\in Q} \prod_{j=1}^{2m} 2^\eta K \left(N''N^{N''}\frac{s^k_j-s^k_{j-1}}{t^k}\lambda([0,t]^{N''})\right)^{-(d+q_j \eta)/2}\\
    &= \left(2^\eta K\right)^{2m} \left(N''N^{N''}\frac{\lambda([0,t]^{N''})}{t^k}\right)^{-m(d+\eta)} \sum_{q\in Q} \prod_{j=1}^{2m}  \left(s^k_j-s^k_{j-1}\right)^{-(d+q_j \eta)/2},
\end{align*}
where we used that $\sum_{j=1}^{2m} q_j = 2m$ in the last equality.\\
Thus, we get, noting $S^k:=\{0=s^k_0\leq s_1^k\leq \dots\leq s_{2m}^k\leq 1\}$ 
\begin{align*}
    J_{\bar{M}}(2m,\eta) &= \int_{\bar{M}^{2m}}\int_{(\R^d)^{2m}}\left| \E e^{i\sum_{j=1}^{2m}\beta_j\cdot \tilde{W}(t^j)} \right|\prod_{j=1}^{2m}|\beta_j|^\eta d\beta_j dt_j \\
    & \leq \prod_{k=1}^{N''} (2m)! \int_{S^k} \left( \left(2^\eta K\right)^{2m} \left(N''N^{N''}\frac{\lambda([0,t]^{N''})}{t^k}\right)^{-m(d+\eta)} \sum_{q\in Q} \prod_{j=1}^{2m}  \left(s^k_j-s^k_{j-1}\right)^{-(d+q_j \eta)/2} \right)^{1/N''} ds^k_j\\
    &\leq (2m)!^{N''} (2^{\eta}K)^{2m}(N''N^{N''})^{-m(d+\eta)} \left( \sum_{q\in Q} \int_S \prod_{j=1}^{2m}  \left(s_j-s_{j-1}\right)^{-(d+q_j \eta)/2N''} \right)^{N''} ds_j.
\end{align*}
Finally, using Lemma \ref{lemma_gamma} we showed
\begin{align*}
    J_{\bar{M}}(2m,\eta) &\leq (2m)!^{N''} C^{2m} \left( \sum_{q\in Q} \frac{\prod_{j=1}^{2m} g\left( 1-\frac{d+q_j\eta}{2N''} \right)}{g\left( 1+2m(1-\frac{d+\eta}{2N''}) \right)} \right)^{N''},
\end{align*}
with $C:=2^{\eta}K(N''N^{N''})^{-\frac{d}{2}}$. Using that $\# Q \leq 3^{2m}$ and that $g\left( 1-\frac{d+q_j\eta}{2N''} \right) \leq g\left( 1-\frac{d+2\eta}{2N''} \right) < \infty$ for any $\eta < N'' - \frac{d}{2}$, we get the desired result.
\end{proof}

\noindent
Moreover, we will need the following result, whose proof can be adapted from the proof of Theorem 5.1 in \cite{berman1969local}.
\begin{lemma}
\label{thm2_1_lemma2}
Fix a projection $p$ and a point $s$ and let $\Bar{W}$ be the associated process. Let $M \subset \R_+^N$ be a bounded rectangle and $\Bar{M}$ its projection by $q$ onto $V$. For $m \in \N$ and $x,y \in \R^d$ close enough, we have for any $0<\alpha_0 <1$ and $0<\eta<\min(1,N''-\frac{d}{2})$ that there exists a constant $C>0$ such that,
$$ \E\left[\left|L_M(x\,;s,p)- L_M(y\,;s,p) \right|^{2m}\right] \leq C (|x-y|^{\alpha_0})^{2m\eta}J_{\bar{M}}(2m,\eta). $$
\end{lemma}

\noindent We are now ready to prove Proposition \ref{thm2_1}. Let $0<\alpha<\frac{1}{2}\min(1,N''-\frac{d}{2})$. Using the Markov inequality, it can be shown that it is enough to prove that there exists $C>0$ such that 
\begin{align}
    \label{thm2_1_obj}
    \E\left(\exp{\frac{|L_M(x\,;s,p)- L_M(y\,;s,p)|}{|x-y|^{2\alpha}}}\right) \leq C.
\end{align}
With Fubini's theorem and using that $\exp=\cosh + \sinh \leq 2\cosh$ and writing the Taylor expansion of the function, we are justified to control only the even terms of the series. Thus, we compute, using Lemma \ref{thm2_1_lemma1} and Lemma \ref{thm2_1_lemma2}, for $0<\alpha_0 <1$ and $0<\eta<\min(1,N''-\frac{d}{2})$,
\begin{align*}
\frac{\E\left[|L_M(x\,;s,p)- L_M(y\,;s,p)|^{2m}\right]}{|x-y|^{4\alpha m}(2m)!} & \leq \dfrac{C^{2m}|x-y|^{2m\alpha_0\eta}(2m)!^{N''}}{|x-y|^{4m\alpha}(2m)!g\left( 1+2m(1-\frac{d+\eta}{2N''}) \right)^{N''}} \\
&\leq C^{2m} |x-y|^{2m(\eta\alpha_0-2\alpha)}\frac{(2m)!^{N''-1}}{g\left( 1+2m(1-\frac{d}{2N''}) \right)^{N''}}.
\end{align*}
With Stirling's formula we can show that the fraction is bounded by a constant $C>0$. Moreover, we can choose the values of $\alpha_0$ and $\eta$ such that $\alpha_1:=\eta\alpha_0-2\alpha>0$. We get,
$$ \E\left(\exp{\frac{|L_M(x\,;s,p)- L_M(y\,;s,p)|}{|x-y|^{2\alpha}}}\right) \leq 2\sum_{m\geq 0} \frac{\E\left[|L_M(x\,;s,p)- L_M(y\,;s,p)|^{2m}\right]}{|x-y|^{4\alpha m}(2m)!} \leq \sum_{m\geq 0} (C(\eta,\alpha_0)|x-y|^{\alpha_1})^{2m}.$$
This sum converges if $x$ and $y$ are close enough, such that \eqref{thm2_1_obj} holds.

\subsection{Proof of Proposition \ref{thm2_2}}
Since the proof can be immediately adapted to the general case, we will prove the theorem for $x=0$. Each $(s,p)$ with $s\in\R^{N'}$ and $p$ an orthogonal projection from $\R^N$ to $\R^{N'}$ can be associated in a one-to-one relation with $v$ an affine sub-spaces of $\R^N$ by the bijection $V=p^{-1}(s)$. This relation allows us to define a well-suited distance with
\begin{align*}
    d\big((s,p),(s',p')\big) := \inf &\Big\{ \max\{ |c_{l,k}-c_{l,k}'| ~:~l=1,\dots,N \text{ and }k=0,\dots,N''\} \\
    :~~&  \forall k=1,\dots N'', ~\sum_{l=1}^N (c_{l,k})^2 = \sum_{l=1}^N (c'_{l,k})^2 = 1, \\
    & p^{-1}(s) = (c_{1,0},\dots,c_{N,0}) + \text{Vect}\left( \sum_{l=1}^N c_{l,1} e_l,\dots,\sum_{l=1}^N c_{l,N''} e_l \right), \\
    & \left. p'^{-1}(s') = (c'_{1,0},\dots,c'_{N,0}) + \text{Vect}\left( \sum_{l=1}^N c'_{l,1} e_l,\dots,\sum_{l=1}^N c'_{l,N''} e_l \right) \right\}.
\end{align*} 
For the rest of the proof and for notational and computational simplicity, we will suppose that $M=[1,2]^N$. Moreover, we will assimilate the metric space of the $(s,p)$ and the metric space of the plans $V=p^{-1}(s)$ associated, noted $\Vv$.\\
For any $n\geq 0$ we introduce the dyadic set 
\begin{align}
    \label{diadic_set_A_n}
    A_n := \left\{ \left(\tfrac{i_{1,0}}{2^m},\dots,\tfrac{i_{N,0}}{2^n}\right) + \text{Vect}\left( \sum_{l=1}^N \tfrac{i_{l,1}}{2^n}e_l,\dots,\sum_{l=1}^N \tfrac{i_{l,N''}}{2^n}e_l\right) ~:~\forall l,k \text{ we have} ~-2^n\leq i_{l,k} \leq 2^n \right\},
\end{align} 
and we can easily see that $\#A_n \leq C2^n$ for a certain constant $C>0$, and that the balls $(B(x,2^{-n}))_{x\in A_n}$ cover $\Vv$.\\
We begin by proving the following lemma
\begin{lemma}
\label{thm2_2_lemma}
For any $0<\alpha<\frac{1}{4}\min(1,N''-\frac{d}{2})$, there exists a constant $c>0$ and a (random) rank $n_0$ such that if $n\geq n_0$, for any couple $(s,p), (s',p') \in A_n$ with $d\big((s,p),(s',p')\big)=2^{-n}$ it holds
	$$\P(|L_M(0\,;s,p)-L_M(0\,;s',p')| \geq (n2^{-n})^\alpha ) \leq  e^{-c(n2^{-n})^\alpha}.$$
\end{lemma}
\begin{proof}
Fix for a moment an element $(s,p)\in\cup_{n\geq 0} A_n$. We note $\Gamma$ the parametrization associated to $V=p^{-1}(s)$, with the notation from Remark \ref{remark_parametrization}. Since we will from now on consider only $(s',p')$'s close to $(s,p)$, we can use the same parametrization $\Gamma$ for the elements $(s',p')$. The only consequence is that the lipschitz constant in Lemma \ref{lemma_parametrization} to be $2N$, which will not change anything to our computations. We will note $\tilde{W}$ the process associated to $V:=p^{-1}(s)$ and $\tilde{W}'$ the process associated to $V':=p'^{-1}(s')$ (both parametrized by $\Gamma$). For notational and computational simplicity, we will suppose that $V=\text{Vect}(e_1,\dots,e_{N''})$ and that $q(M)=:\bar{M}=[1,2]^{N''}$.\\
With Corollary \ref{cor_parametrization}, we can show that there exists a (random) rank $n_0$ and a constant $C_0>0$ such that almost surely for every $t\in\bar{M}$,
\begin{align}
\label{modulus_continuity_thm2}
    |W(t)-\tilde{W}(t) | < C_0n2^{-n/2}. 
\end{align}
Let $\delta:=C_0n2^{-n/2}$ and $\epsilon=2\delta$, from \eqref{modulus_continuity_thm2} we can deduce that almost surely, for any $t \in \bar{M}$,
$$\{\tilde{W}(t)\in B(0,\epsilon-\delta)\} \subset \{W(t)\in B(0,\epsilon)\}.$$
Then from the definition of local time,
$$\int_{B(0,\epsilon-\delta)} L_M(x\,;s',p')d x \leq \int_{B(0,\epsilon)} L_M(x\,;s,p)d x,$$
which implies that
$$ \inf_{x\in B(0,\epsilon)}L_M(x\,;s',p')\leq \inf_{x\in B(0,\epsilon-\delta)}L_M(x\,;s',p')\leq 2^d\sup_{x\in B(0,\epsilon)}L_M(x\,;s,p).$$
With the same approach, one can also show that
$$ \inf_{x\in B(0,\epsilon)}L_M(x\,;s,p)\leq \inf_{x\in B(0,\epsilon-\delta)}L_M(x\,;s,p)\leq 2^d\sup_{x\in B(0,\epsilon)}L_M(x\,;s',p').$$
Thus we have
$$ |L_M(0\,;s,p)-L_M(0\,;s',p')| \leq 2^d\left[\sup_{x\in B(0,\epsilon)} |L_M(x\,;s,p)-L_M(0\,;s,p)| + \sup_{x\in B(0,\epsilon)} |L_M(x\,;s',p')-L_M(0\,;s',p')|\right].$$
Therefore, let $0<\alpha<\frac{1}{2}\min(1,N''-\frac{d}{2})$ and we compute for $n\geq n_0$,
\begin{align*}
    &\P\left( |L_M(0\,;s,p)-L_M(0\,;s',p')| \geq \frac{2^{d+1}}{2^{\alpha'}-1}(2C_0n2^{-n/2})^\alpha \right) \\
    \leq ~ & \P\left( \sup_{x\in B(0,\epsilon)} |L_M(x\,;s,p)-L_M(0\,;s,p)| + \sup_{x\in B(0,\epsilon)} |L_M(x\,;s',p')-L_M(0\,;s',p')| \geq \frac{2}{2^{\alpha'}-1
    }\epsilon^\alpha\right) \\
    \leq ~ & \P\left( \sup_{x\in B(0,\epsilon)} |L_M(x\,;s,p)-L_M(0\,;s,p)|\geq \epsilon^\alpha \right)+\P\left(\sup_{x\in B(0,\epsilon)} |L_M(x\,;s',p')-L_M(0\,;s',p')| \geq  \frac{1}{2^{\alpha'}-1}\epsilon^\alpha \right).
\end{align*}
We will only treat the left term since the right term can be handled similarly, and thus for notational simplicity we allow ourselves to drop the dependencies in $M,s$ and $p$. We define for any $k\geq 0$ the dyadic set $B_k:=\{\epsilon(\frac{i_1}{2^k},\dots,\frac{i_d}{2^k}) ~:~-2^k\leq i_1,\dots,i_d\leq 2^k\}$, and we fix $\alpha<\alpha'<\frac{1}{2}\min(1,N''-\frac{d}{2})$. Notice that from Lemma \ref{thm2_1_lemma2} we can apply Kolmogorov's continuity criterion to obtain continuity of (at least for a modification) $x\mapsto L(x)$, such that we have
\begin{align*}
    \left\{ \sup_{x\in B(0,\epsilon)} |L(x)-L(0)| \geq  \frac{\epsilon^\alpha}{2^{\alpha'}-1} \right\} \subset \left\{ \bigcup_{k\geq 0} \bigcup_{x,y \in B_k \atop |x-y|=\epsilon 2^{-k}} |L(x)-L(y)|\geq \epsilon^\alpha 2^{-\alpha'k} \right\}.
\end{align*}
Indeed, fix $x\in B(0,\epsilon)$ and $(x_k)_{k\geq 0}$ a sequence such that $x_k\in B_k$, $|x_{k+1}-x_k|\leq \epsilon 2^{-(k+1)}$ and $|x-x_k|\leq \epsilon 2^{-k}$. If $\forall k\geq 0,\,\forall x,y \in B_k$ such that $|x-y|=\epsilon 2^{-k}$ it holds $|L(x)-L(y)| < \epsilon^\alpha 2^{-\alpha'k}$, then we obtain 
\begin{align*}
    |L(x)-L(0)| \leq \sum_{k=0}^\infty |L(x_{k+1})-L(x_k)| < \sum_{k=0}^\infty \epsilon^\alpha 2^{-\alpha'(k+1)} = \frac{\epsilon^\alpha}{2^{\alpha'}-1}.
\end{align*}
Thus, we have with Proposition \ref{thm2_1}
\begin{align*}
    \P\left( \sup_{x\in B(0,\epsilon)} |L(x)-L(0)| \geq  \frac{\epsilon^\alpha}{2^{\alpha'}-1} \right) &\leq \P\left(\bigcup_{k\geq 0} \bigcup_{x,y \in B_k \atop |x-y|=\epsilon 2^{-k}} |L(x)-L(y)|\geq \epsilon^\alpha 2^{-\alpha'k}\right)\\
    &\leq \P\left(\bigcup_{k\geq 0} \bigcup_{x,y \in B_k \atop |x-y|=\epsilon 2^{-k}} |L(x)-L(y)| \geq (\epsilon 2^{-k})^\alpha\right)\\
    &\leq \sum_{k\geq 0} 2d2^{kd} \max_{x,y \in B_k \atop |x-y|=\epsilon2^{-k}} \P\left(|L(x)-L(y)|\geq (\epsilon2^{-k})^\alpha\right) \leq \sum_{k\geq 0} 2d2^{kd} e^{-c(2^{-k}\epsilon)^{-\alpha}}.
\end{align*}
We can find a constant $c'>0$ such that $\sum_{k\geq 0} 2d2^{kd} e^{-c(2^{-k}\epsilon)^{-\alpha'}} \leq e^{-c'\epsilon^{-\alpha}}$ and rebranding the constants we obtain the desired result.
\end{proof}

\noindent Now we are ready to prove Proposition \ref{thm2_2}. Let's fix $(s,p)$ and for $n\geq 1$ choose any sequence $(s_n,p_n)\in A_n$ such that $d\big( (s,p),(s_n,p_n)\big) \leq 2^{-n}$. We compute, for $0<\alpha<\frac{1}{2}\min(1,N''-\frac{d}{2})$, using Lemma \ref{thm2_2_lemma}
$$\sum_{n\geq 1} \P(|L_M(0\,;s,p)-L_M(0\,;s_n,p_n)|\geq 2^{-n\alpha}) \leq \sum_{n\geq 1} e^{-c2^{-n\alpha}} <\infty. $$
Thus, using Borel-Cantelli first lemma we can deduce that $\P(\limsup_{n\geq 1} |L_M(0\,;s,p)-L_M(0\,;s_n,p_n)|\geq 2^{-n\alpha}) =0$, which implies that indeed 
$$\P(L_M(0\,;s_n,p_n) \mapsto L_M(0\,;s,p),~\text{for }n\mapsto \infty)=1.$$

\subsection{Proof of Theorem \ref{thm2}}
\label{section_thm2_final}
In this section we show why Proposition \ref{thm2_2} implies Theorem \ref{thm2}. Since $\Vv$ is locally compact, we can suppose without loss of generality that we $(s,p)$ are restricted to a compact subset $K$ of $\Vv$ (see proof of Proposition \ref{thm2_2} for details about the notations).\\
First, fix $(s_0,p_0)\in K$ and let $0<\varepsilon<1$. Note $M_0:=[1,2]^N$ and suppose that the $a>0$ is small enough for 
$$\P(L_{M_0}(0\,;s_0,p_0) > a) > \varepsilon.$$ 
Secondly, fix $0<\alpha<\frac{1}{4}\min(1,N''-\frac{d}{2})$ and from Lemma \ref{thm2_2_lemma} we have that, for any rank $N_0$,
\begin{align*}
    &\P\left( \bigcup_{n\geq N_0} \bigcup_{(s,p),(s',p')\in A_n \atop d((s,p),(s',p')=2^{-n}} |L_{M_0}(0\,;s,p)-L_{M_0}(0\,;s',p')|\geq (n2^{-n})^\alpha \right) \\
    \leq~ & \sum_{n\geq N_0} \P\left( \bigcup_{(s,p),(s',p')\in A_n \atop d((s,p),(s',p')=2^{-n}} |L_{M_0}(0\,;s,p)-L_{M_0}(0\,;s',p')|\geq (n2^{-n})^\alpha \right)\\
    \leq~ & \sum_{n\geq N_0} e^{-c(n2^{-n})^\alpha} <\infty.
\end{align*} 
Thus, we obtained that if we suppose $N_0$ big enough
\begin{align*}
    \P\left( \left\{L_{M_0}(0\,;s_0,p_0) > a\right\}~~\bigcup ~~ \left\{\bigcap_{n\geq N_0} \bigcap_{(s,p),(s',p')\in A_n \atop d((s,p),(s',p')=2^{-n}} |L_{M_0}(0\,;s,p)-L_{M_0}(0\,;s',p')| < (n2^{-n})^\alpha\right\} \right) > \frac{3}{4}\varepsilon.
\end{align*}
One should note that the event inside the probability implies that $L_{M_0}(0\,;s,p) > a - \sum_{n\geq N_0} (n2^{-n})^\alpha$ for all $(s,p)\in \cup_{n\geq 0}$ such that $d((s_0,p_0),(s,p)) \leq 2^{-N_0}$. Fix $0<\varphi<a$, and we can always suppose $N_0$ big enough for $\varphi<a-\sum_{n\geq N_0} (n2^{-n})^\alpha$. Thus, we showed that, for any $0<\varepsilon<1$ and $\varphi>0$, there exists a rank $N_0$ such that
$$\P\big(\forall\,(s,p)\in  \cup_{n\geq 0} A_n,\text{ such that }d((s_0,p_0),(s,p))\leq 2^{-N_0},~L_{M_0}(0\,;s,p) > \varphi\big)>\frac{3}{4}\varepsilon.$$
Thus, with Proposition \ref{thm2_2} we can deduce that
$$\P\big(\forall\,(s,p)\in  K,\text{ such that }d((s_0,p_0),(s,p))\leq 2^{-N_0},~L_{M_0}(0\,;s,p) \geq \varphi\big)>\frac{3}{4}\varepsilon.$$
In particular, we proved that for any $(s_0,p_0)\in K$, there exists an open ball $B(s_0,p_0)\subset \Vv$ centered in $(s_0,p_0)$ and of radius $r(s_0,p_0)$ such that 
\begin{align}
    \label{thm2_proof_eq1}
    \P(\forall (s,p)\in B(s_0,p_0),~L_{M_0}(0\,;s,p)\geq 1)>\frac{2}{3}.
\end{align}
We note $E_0$ the event in \eqref{thm2_proof_eq1}, $(B_t)_{t\geq 0}$ the process defined by $B_t := W([0,t])^N$ and we suppose $b>0$ big enough for
\begin{align}
    \label{thm2_proof_eq2}
    \P(\inf(t\geq \tfrac{1}{2}~:~|B_t|\leq b\sqrt{t}) ~\leq 1)>\frac{2}{3}.
\end{align}
We note $F_0$ the event in \eqref{thm2_proof_eq2} and we notice that $\P(E_0\cap F_0)>\frac{1}{3}$. Moreover, for any $n\geq 1$ we define $M_n:=[2^{2n},2^{2n+1}]^N$, and the events
\begin{align*}
    E_n:= \big\{\forall (s,p)\in B(s_0,p_0),~L_{M_n}(0\,;s,p)\geq 2^{2nN} \big\} \quad \text{and} \quad F_n := \big\{ \inf(t\geq 2^{2n-1}~:~|B_t|\leq b\sqrt{t}) ~\leq 2^{2n} \big\}.
\end{align*}
From the scaling property of the brownian sheet we deduce that $\P(E_n~|~F_n) \geq \P(E_n\cap F_n) = \P(E_0\cap F_0) > \frac{1}{3}$. Considering the filtration $(\Ff_n)_{n\geq 0}$ defined by $\Ff_n := \sigma(W(A) ~:~A\subset [0,2^{2n}]^N)$, and since almost surely $F_n$ occurs infinitely often, we can apply Lemma \ref{lemma_borel_cantelli} to conclude that $\P(\limsup E_n) =1$. Moreover, since the event $E_n$ implies that for all $(s,p)$ in $B(s_0,p_0)$, we have $p^{-1}(s) \cap \Zz \neq \emptyset$, we proved that 
$$\P(\forall (s,p)\in B(s_0,p_0),~p^{-1}(s) \cap \Zz \neq \emptyset)=1.$$
Finally, $K$ is compact and thus can be covered by a finite number of balls $B(s_0,p_0)$. This gives the desired result.

\newpage
\appendix

\section{Miscellaneous lemmas}

We begin by a key result from \cite[Theorem 2.4]{orey1973sample} which will often be used in this section. 
\begin{lemma}
\label{lem_OreyPruitt_original}
If $W$ is a $(N,d)$-Brownian sheet, there exists a certain random rank $n_0$, such that for every $n\geq n_0$, we have that $|W(\mathbf{t})-W(\mathbf{s})|\leq n2^{-n/2}$ as long as $\|\mathbf{t}-\mathbf{s}\|_\infty<2^{-n}$.
\end{lemma}

\subsection*{Lemmas for Theorem \ref{thm0}}

Here we write some useful secondary Lemmas for the proof of Theorem \ref{thm0}. The notations are the one given in Section \ref{section_thm0}. The following can be directly proven from Lemma \ref{lem_OreyPruitt_original}.

\begin{lemma}
\label{lem_OreyPruitt}
There exist a rank $N$ (random) such that for every $n\geq N$, we have 
$$\left\{S^n_{ij} \text{ is good}\right\} \subset \left\{ |W(\mathbf{t})|<n2^{-n/2}, ~~ \forall \mathbf{t}\in S^n_{ij}\cap\Gamma^n_i\right\}. $$
\end{lemma}


\begin{lemma}
\label{lem_BM_line}
If $p~:~\R^2\to\R$ is a linear map such that $\ker(p)\cap\dot{\R}^2_+\neq\{0\}$, then for every line $\Gamma^n_i$, there exist a Brownian motion $B^n_i$ such that the Brownian sheet along the line $\Gamma^n_i$ is equal to the Brownian motion $B^n_i$ up to a deterministic time change, \textit{ie} that if $\mathbf{t}=(t_1,t_2)\in\Gamma^n_i$,
$$W(\mathbf{t})=W(t_1,\alpha t_1+\beta^n_i)=B^n_i(t_1(\alpha t_1+\beta^n_i)).$$
\end{lemma}

\begin{proof}
In this proof, we are making an exception and we are considering that the Brownian sheet $W$ is defined on all $\R^2_+$.\\
Thanks to our assumption on $p$, we have $a>0$ and the process $(\tilde{W}_t)_{t\geq 0}$ defined by $\tilde{W}_t:=W(t,\alpha t+\beta^n_i)$ is obviously a martingale such that $\tilde{W}_0=0$. Thus, by Dubins-Schwarz theorem, we obtain that there exists a Brownian motion $B^n_i$ such that $\tilde{W}_t=B^n_i(\langle \tilde{W}\rangle_t)$. Finally, the result holds since indeed $\langle \tilde{W}\rangle_t=t(\alpha at+\beta^n_i)$.
\end{proof}

\begin{lemma}
\label{lem_Davis}
Let $B$ be a Brownian motion in $\R^2$, and take constants $0<r<R$. Define $\tau:=\tau_R\wedge\tau_r$ where $\tau_r:=\inf\{t\geq 0~:~|B(t)|=r\}$.\\
Then for every $a\in\R^2$ such that $r<|a|<R$, we have that $p_a:=\P_a(|B_\tau|=R)$ checks
$$p_a = \dfrac{\ln|a|-\ln r}{\ln R-\ln r}.$$
\end{lemma}

\begin{proof}
Since the function $\ln|\cdot|$ is harmonic on $\{a\in\R^2~:~r<|a|<R\}$ and $\tau$ is almost surely finite, we can use \cite[Thm 2.1]{davies} and deduce that 
$$\ln|a|=\E_a\ln|[B_\tau|]=p_a \ln R + (1-p_a)\ln r.$$
Thus, by isolating $p_a$ we get the desired result.
\end{proof}

\subsection*{Lemmas for Theorem \ref{thm1}}
Here we write some useful secondary Lemmas for the proof of Theorem \ref{thm1}. The notations are the one given in Section \ref{section_thm1}.\\
We begin with a Lemma that allows us to get a modulus of continuity as in Lemma \ref{lem_OreyPruitt} but uniformly in $K$.

\begin{lemma}
\label{lemma_parametrization}
Taking the notations from Remark \ref{remark_parametrization}, for every linear subspace $K\subset \R^N$, there exists $i_0\leq \binom{N}{N''}$ such that the parametrization $\Gamma_{i_0}$ is a $N$-lipschitz application.
\end{lemma}

\begin{proof}
First, one should note that for all vector $v\in\R^N$, there exist $i\leq N$ such that $|\cos(v,e_i)|\geq \frac{1}{\sqrt{N}}$. Thus, one can choose $i_0$ such that for every $e_j \in V_{i_0}$, we have $|\cos(\Gamma_{i_0}(e_j),e_j)|\geq \frac{1}{\sqrt{N}}$, which is equivalent to $|\Gamma_{i_0}(e_j)| \leq \sqrt{N}$, by the following iterative procedure :\\
For $0\leq n \leq N$ we note $\R^n:=\text{Vect}(e_1,\dots,e_n)$ and we proceed with an iteration on $1\leq n \leq N''$. In this iteration, the $f_i$'s will denote the vectors included in $V_{i_0}$, and the $f'_i$'s will denote the vectors that will not be included in $V_{i_0}$.
\begin{itemize}
    \item \textit{when} $n=1$ : If $\dim(R^n \cap K)>0$, note $f_1=e_1$, and otherwise note $f'_1=e_1$.
    \item \textit{when $n>1$ and the $f_i$'s and $f_i'$'s are given for $i\leq n-1$} :\\ If $\dim(\R^n \cap K)>\dim(\R^{n-1}\cap K)$, note $v$ the orthogonal projection of $e_n$ on $K\cap\R^n$ orthogonally to the $f_i$'s and $e_n$ in $\R^n$. Then define $f_n$ the element $e_i$ such that $|\cos(v,e_i)|\geq \frac{1}{\sqrt{N}}$. By construction $i\leq n$, and more specifically either $f_n=e_n$, either $f_n$ was in the $f'_i$'s, in which case we note $f'_n:=e_n$.\\
    If $\dim(\R^n \cap K) = \dim(\R^{n-1}\cap K)$, note $f'_n = e_n$.
\end{itemize}
After this iteration, since at each step $n$ we have $0\leq \dim(\R^n \cap K)-\dim(\R^{n-1}\cap K)\leq 1$, we have a family of $N''$ vectors $f_i$'s. If we note $i_0$ the number such that $V_{i_0} = \text{Vect}(f_i)$, we have that by construction $|\cos(\Gamma_{i_0}(f_i),f_i)|\geq \frac{1}{\sqrt{N}}$. Indeed, at every step $n$, if $f_n = e_n$ the orthogonal projection considered is the restriction to $\R^n$ of the orthogonal projection $\Gamma_{i_0}$ and we have by construction that $|\cos(v,f_n)|\geq \frac{1}{\sqrt{N}}$. While if $f_n$ was in the $f'_i$'s, $v$ is an orthogonal projection of $e_n$ on $K$ parallel to $f_n$ and $\Gamma_{i_0}(f_n)$ is an orthogonal projection of $f_n$ on $K$ parallel to $e_n$, thus the $\cos$ are the same and $|\cos(\Gamma_{i_0}(f_n),f_n)|\geq \frac{1}{\sqrt{N}}$ indeed also hold.\\
Finally, since $\Gamma_{i_0}$ is linear, we can deduce that for any $v:=\sum_{e_j\in V_{i_0}} c_j e_j$ such that $|v|^2=\sum_{e_j\in V_{i_0}} |c_j|^2=1$,
\begin{align*}
    \left| \Gamma_{i_0}(v) \right| &= \left|\sum_{e_j\in V_{i_0}} c_j \Gamma_{i_0}(e_j)\right| \leq \sum_{e_j\in V_{i_0}} |c_j| \left|\Gamma_{i_0}(e_j)\right| \leq \sqrt{N} \sum_{e_j\in V_{i_0}} |c_j| \leq N,
\end{align*}
where we used Hölder's inequality in the last inequality. 
\end{proof}

Using Lemma \ref{lem_OreyPruitt_original}, we can deduce from Lemma \ref{lemma_parametrization} the following result.
\begin{cor}
\label{cor_parametrization}
Taking the notations from Remark \ref{remark_parametrization}, there exists a rank $n_0$ (random) such that for every linear subspace $K$ of dimension $N''$, there exists $i_0\leq \binom{N}{\dim(K)}$ such that if $t,s \in V_{i_0}$ and $\|t-s\|_\infty \leq 2^{-n_0}$, we have
$$ |\tilde{W}_{i_0}(t) - \tilde{W}_{i_0}(s)| < n2^{-n/2}. $$
\end{cor}

Finally, we need the following result.
\begin{lemma}
\label{lemme_tx}
Let $x>0$, and $T_x$ defined in \eqref{def_tx}. Noting the number of points in $T_x$ by $\# T_x$, we have
$$\#T_x \leq (r2^{n+1}x)^{N''} \qquad \text{and} \qquad \# T_0 \leq (2r)^{N''}.$$
\end{lemma}

\begin{proof}
Suppose that we are uniformly drawing a $t\in T$ and we want to compute $\P(t\in T_x)$. By Laplace, since it's like asking for every coordinate $i=1,\dots,N''$ to be at most $x$-close from one of the $s\in T_q$, we have
$$\P(t\in T_x) \leq \left(\dfrac{2rx}{2^n}\right)^{N''}.$$
Thus, $\# T_x \leq (2rx)^{N''}$.\\
Similar reasoning can be followed to compute $\# T_0$.
\end{proof}

\subsection*{Lemmas for Theorem \ref{thm2}}
A straightforward computations leads to
\begin{lemma}
\label{lemma_gamma}
We have, for any $m\in \N$ and any family $(b_j)_{j\leq 2m}$ where $0<b_j<1$,
$$	
\int_{0=s_0\leq s_1\leq \dots \leq s_{2m}\leq 1}\prod_{j=1}^{2m}(s_j-s_{j-1})^{-b_j} d s_j =\frac{\prod_{j=1}^{2m}g(1-b_j)}{g(1+2m-\sum_{j=1}^{2m}b_j)},
$$
where $g$ is the Euler Gamma-function.
\end{lemma}

\begin{lemma}
\label{lemma_borel_cantelli}
    Let $(A_n)_{n\geq 0}$ be a sequence of event from a certain filtered probability space $(\Omega,\Ff,(\Ff_n)_{n\geq 0},\P)$ with $A_n$ being measurable with respect to $\Ff_n$.\\
    Suppose that that there exists a constant $c>0$ such that there exists infinitely many $n\geq 0$ such that almost surely $\P(A_n~|~\Ff_n) > c$.\\
    Then $P(\limsup A_n)=1$.
\end{lemma}

\newpage
\bibliography{reference}

\end{document}